\documentclass[12pt]{article}
\usepackage{layout}
\usepackage{url}
\usepackage{proof}
\usepackage{mathrsfs}
\usepackage{amsmath,amsthm,amssymb}
\usepackage{cleveref}
\usepackage[dvipdfmx]{graphicx}
\usepackage[all]{xy}
\usepackage{color}
\usepackage{indentfirst}

\numberwithin{equation}{section}

\newtheorem{thm}{Theorem}[section]
\newtheorem{prop}[thm]{Proposition}
\newtheorem{lemm}[thm]{Lemma}
\newtheorem{cor}[thm]{Corollary}
\newtheorem{defi}[thm]{Definition}

\newcommand{\Qlb}{\overline{\mathbb{Q}}_\ell}

\newcommand{\Hom}{\mathrm{Hom}}

\newcommand{\Perv}{\mathrm{Perv}}
\newcommand{\Rep}{\mathrm{Rep}}

\newcommand{\LocSys}{\mathrm{LocSys}}

\newcommand{\Sat}{\mathrm{Sat}}

\newcommand{\Hck}{\mathcal{H}\mathrm{ck}}

\newcommand{\Div}{\mathrm{Div}}
\newcommand{\ULA}{\mathrm{ULA}}
\newcommand{\Perf}{\mathrm{Perf}}

\newcommand{\et}{\mathrm{\acute{e}t}}
\newcommand{\Det}{D_{\mathrm{\acute{e}t}}}
\newcommand{\Sh}{\mathrm{Sh}}
\newcommand{\Shet}{\mathrm{Sh}_{\mathrm{\acute{e}t}}}

\newcommand{\Gr}{\mathrm{Gr}}

\newcommand{\Wdisj}{\text{$W$-disj}}

\DeclareMathOperator{\Spec}{Spec}
\DeclareMathOperator{\Spd}{Spd}

\title{\Large{Two monoidal structures on Satake category in mixed characteristic}}
\author{Katsuyuki Bando\footnote{Katsuyuki Bando, Research Institute for Mathematical Sciences, Kyoto University, Kyoto, Japan, \texttt{kbando@kurims.kyoto-u.ac.jp}}}

\begin{document}
\maketitle
\abstract{Fargues and Scholze proved the geometric Satake equivalence over the Fargues--Fontaine curve.
This can be transferred to the geometric Satake equivalence concerning a Witt vector affine Grassmannian via a nearby cycle.
On the other hand, Zhu proved the geometric Satake equivalence concerning a Witt vector affine Grassmannian.
In this paper, we explain the coincidence of these two geometric Satake equivalences, including the coincidence of the two symmetric monoidal structures on the Satake category.}
\section{Introduction}
The geometric Satake equivalence is the equivalence between the category of equivariant perverse sheaves on the affine Grassmannian of a reductive group and the category of finite-dimensional algebraic representations of its dual group.
This equivalence connects geometric objects with representation-theoretic objects, so it is important in geometric representation theory.
In particular, it is closely related to the geometric Langlands program.
Therefore, considering the geometric Satake equivalence in mixed characteristic is also important in number theory.
In \cite{Zhumixed}, Zhu proved the geometric Satake equivalence in mixed characteristic using a ring of Witt vectors.
This plays an important role in various representation theoretic or number theoretic results.
On the other hand, in \cite{FS}, Fargues--Scholze proved the geometric Satake equivalence in mixed characteristic using the Fargues--Fontaine curve, in a form that makes its relation to the categorical Langlands conjecture clear.
We discuss the relation between Zhu's and Fargues--Scholze's geometric Satake equivalence.
In this way, various results proved using Zhu's equivalence can be transferred to Fargues--Scholze's context.

Let $F$ be a $p$-adic local field with a residue field $\mathbb{F}_q$.
Write $\mathcal{O}_F$ for its ring of integers.
Put $k:=\overline{\mathbb{F}}_q$.
Let $G$ be a reductive group scheme over $\mathcal{O}_F$.
Let $\ell$ be a prime not equal to $p$.
Write $\widehat{G}$ for the Langlands dual group of $G$ over $\Qlb$.

In \cite{Zhumixed}, Zhu defined a Witt vector affine Grassmannian \[
\Gr_G=\Gr^{\mathrm{Witt}}_{G,\Spec k}
\]
as a geometrization of $G(W(k)\otimes_{W(\mathbb{F}_q)}F)/G(W(k)\otimes_{W(\mathbb{F}_q)}\mathcal{O}_F)$, which is an ind-projective ind-perfect scheme over $k$.
Let $\Perv_{L^{+,\mathrm{Witt}}G}(\Gr_{G,\Spec k}^{\mathrm{Witt}},\Qlb)$ be the category of $L^{+,\mathrm{Witt}}G$-equivariant perverse sheaves, where $L^{+,\mathrm{Witt}}G$ is a positive loop group over $k$, a geometrization of $G(W(k)\otimes_{W(\mathbb{F}_q)}\mathcal{O}_F)$.
Let $\Rep(\widehat{G},\Qlb)$ be the category of finite dimensional representations of $\widehat{G}$ over $\Qlb$.
According to \cite{Zhumixed}, there is a symmetric monoidal structure on $\Perv_{L^{+,\mathrm{Witt}}G}(\Gr_{G,\Spec k}^{\mathrm{Witt}},\Qlb)$ and there is a symmetric monoidal equivalence called the geometric Satake equivalence:
\[
\Perv_{L^{+,\mathrm{Witt}}G}(\Gr_{G,\Spec k}^{\mathrm{Witt}},\Qlb)\overset{\sim}{\to} \Rep(\widehat{G},\Qlb).
\]
Using notions of diamonds, we can rephrase this result as a symmetric monoidal equivalence
\begin{align}\label{eqn:geomSatSpdk}
\Sat(\Hck_{G,\Spd k},\Qlb)\overset{\sim}{\to} \Rep(\widehat{G},\Qlb),
\end{align}
where $\Hck_{G,\Spd k}$ is a local Hecke stack over $\Spd k$, and $\Sat$ denotes a Satake category, see \cite[\sc{Definition} VI.1.6, \S VI.7.1]{FS}.
On the other hand, there is a symmetric monoidal equivalence
\begin{align}\label{eqn:SatSpdC}
\Sat(\Hck_{G,\Spd C},\Qlb)\overset{\sim}{\to} \Rep(\widehat{G},\Qlb),
\end{align}
where $C$ is the completion of an algebraic closure of $F$.
This can be proved by the same argument as the geometric Satake equivalence in \cite[\S VI]{FS}.
More precisely, the construction of the symmetric monoidal structure of the fiber functor can be proved by a parallel argument to the case of $\Div^1$ in \cite[\S VI]{FS}:
\begin{prop}[Proposition \ref{prop:F1monoidal}]\label{prop:F1monoidalintro}
There exists a canonical symmetric monoidal structure of 
\[
F_{\Spd C}^1:=H^*(\Gr_G,-)\colon \Sat(\Hck_{G,\Spd C},\Qlb)\to \mathrm{Vect}_{\Qlb},
\]
where $\mathrm{Vect}_{\Qlb}$ is the category of finite dimensional $\Qlb$-vector spaces.
\end{prop}
The identification of the Tannakian group with $\widehat{G}$ follows from the case of $\Div^1$, and we obtain (\ref{eqn:SatSpdC}).
By transferring this via a nearby cycle, we have another symmetric monoidal structure on $\Sat(\Hck_{G,\Spd k},\Qlb)$ and another symmetric monoidal equivalence of the same form as (\ref{eqn:geomSatSpdk}), see \cite{B}.
The main theorem is the following:
\begin{thm}\label{thm:main}
The symmetric monoidal structure on the category 
\[
(\Sat(\Hck_{G,\Spd k},\Qlb),\star)
\]
in \cite{B} coincides with the one in \cite{Zhumixed}.
Moreover, the geometric Satake equivalence
\[
\Sat(\Hck_{G,\Spd k},\Qlb)\overset{\sim}{\to} \Rep(\widehat{G},\Qlb)
\]
in \cite{B} coincides with one in \cite{Zhumixed}.
\end{thm}
The fact that Fargues--Scholze's geometric Satake equivalence gives a new proof of Zhu's geometric Satake equivalence is mentioned in \cite[\sc{Remark} I.2.14]{FS}, but it is not mentioned how to show the compatibility of the two symmetric monoidal structures.
This is a nontrivial question since the methods of constructing the two symmetric monoidal structures are quite different. We show this compatibility (and also Proposition \ref{prop:F1monoidalintro}) by using the following lemma:
Put
\[
(\Spd C)_{\Wdisj}^2:=(\Spd C)^2\setminus (\Spd C\times_{\Div^1}\Spd C).
\]
Note that $\Spd C\times_{\Div^1}\Spd C\cong \Spd C\times W_F$ holds since $\Div^1$ is the quotient of $\Spd C$ by the action of the Weil group $W_F$ of $F$ in \cite[\S IV.7]{FS}.
\begin{lemm}\label{lemm:ffjneqintro}
Let
\[
j_{\Wdisj,(\Spd C)^2}\colon (\Spd C)_{\Wdisj}^2\hookrightarrow (\Spd C)^2
\]
be the open immersion.
Let $A\in \Shet((\Spd C)^2,\Lambda)$ be a constant sheaf of a finitely generated $\Lambda$-module, where $\Shet(-,\Lambda)$ is the full subcategory of $\Det(-,\Lambda)$ consisting of complexes concentrated in degree 0.
Then the natural map
\[
A \to R^0j_{\Wdisj,(\Spd C)^2*}j_{\Wdisj,(\Spd C)^2}^*A
\]
is an isomorphism.
\end{lemm}
By this lemma, the calculation of each monoidal structure reduces to the calculation over $(\Spd C)^2_{\Wdisj}$, and this can be done using some exterior products.

Consequently, our result gives a geometric construction of the commutativity constraint of the monoidal structure in \cite{Zhumixed}, which is constructed in \cite{Zhumixed} using a numerical result for the affine Hecke algebra.
\section*{Acknowledgements}
The author is really grateful to his advisor Naoki Imai for useful discussions and comments.
He also thanks the referees for their careful reading of the paper and valuable suggestions and comments.

\section{Notation}
\subsection{Notation on diamonds}
Throughout this paper, $F$ is a $p$-adic local field with a residue field $\mathbb{F}_q$.
Write $\mathcal{O}_F$ for its ring of integers.
Put $k:=\overline{\mathbb{F}}_q$.
We write $C$ for the completion of an algebraic closure of $F$, and $\breve{F}$ for the completion of the maximal unramified extension of $F$ in $C$.

Let $\Perf_k$ be a category of perfectoid spaces over $k$.
When we consider diamonds, or more generally, v-stacks, we work on $\Perf_k$.
As in \cite[\sc{Definition} VI.1.1]{FS}, put
\begin{align*}
\Div^d_{\mathcal{Y}}&:=(\Spd \mathcal{O}_{\breve{F}})^d,\\
\Div^d&:=(\Spd \breve{F}/\varphi^{\mathbb{Z}})^d/\Sigma_d,
\end{align*}
where $\Sigma_d$ is the symmetric group and $\varphi$ is the Frobenius.
Let $G$ be a reductive group over $\Spec \mathcal{O}_F$.
For a small v-stack $S$ with a map $S\to \Div^d$ or $S\to \Div^d_{\mathcal{Y}}$, we get the corresponding divisor $D_S$ on the Fargues--Fontaine curve $X_S$ or $\mathcal{Y}_S$, respectively, which is defined for example in \cite[\S 2]{FS}.
Let $B^+_{\Div^d}(S)$ (resp. $B^+_{\Div^d_{\mathcal{Y}}}(S)$) be the global section of the completion of $X_S$ (resp. $\mathcal{Y}_S$) along $D_S$.
Put $B_{\Div^d_{(\mathcal{Y})}}(S)=B^+_{\Div^d_{(\mathcal{Y})}}(S)[1/\mathcal{I}_S]$, where $\mathcal{I}_S$ is the ideal corresponding to $D_S$.
Then the following spaces are defined as in \cite[\S VI.1]{FS}:\\
The positive loop space and the loop space of $G$ over $\Div^d_{(\mathcal{Y})}$ are defined by
\begin{align}
L^+_{\Div^d_{(\mathcal{Y})}}G\ \colon\ \Perf_k\to \mathrm{Sets},\ S\mapsto G(B^+_{\Div^d_{(\mathcal{Y})}}(S)),\\
L_{\Div^d_{(\mathcal{Y})}}G\  \colon\ \Perf_k\to \mathrm{Sets},\ S\mapsto G(B_{\Div^d_{(\mathcal{Y})}}(S)).
\end{align}
For a v-stack $S$ over $\Div^d_{(\mathcal{Y})}$, put $L^+_S:=L^+_{S/\Div^d_{(\mathcal{Y})}}G:=L^+_{\Div^d_{(\mathcal{Y})}}G\times_{\Div^d_{(\mathcal{Y})}}S$ and $L_S:=L_{S/\Div^d_{(\mathcal{Y})}}G:=L_{\Div^d}G\times_{\Div^d_{(\mathcal{Y})}}S$.
The Beilison--Drinfeld affine Grassmannian of $G$ over $S$ is 
\[
\Gr_{G,S}=\Gr_{G,S/\Div^d_{(\mathcal{Y})}}:=[L_SG/L^+_SG],
\]
and the local Hecke stack of $G$ over $S$ is 
\[
\Hck_{G,S}=\Hck_{G,S/\Div^d_{(\mathcal{Y})}}:=[L^+_SG\backslash L_SG/L^+_SG].
\]
Moreover, we also define the $m$-th congruence subgroup $L^+_{\Div^d_{(\mathcal{Y})}}G^{\geq m}$ as the subgroup of  $L^+_{\Div^d_{(\mathcal{Y})}}G$ corresponding to the kernel of the homomorphism
\[
G(B^+_{\Div^d_{(\mathcal{Y})}}(S))\to G(B^+_{\Div^d_{(\mathcal{Y})}}(S)/\mathcal{I}_S^{m+1}).
\]
Put $L^+_SG^{\geq m}=L^+_{\Div^d_{(\mathcal{Y})}}G^{\geq m}\times_{\Div^d_{(\mathcal{Y})}}S$.
\subsection{Notation on categories}
Let $\ell$ be a prime number not equal to $p$.
Let $\Lambda$ be either $\Qlb$, a finite extension $L$ of $\mathbb{Q}_{\ell}$, its ring of integers $\mathcal{O}_L$, or its quotient ring.
If $\Lambda$ is a torsion ring, then the derived category $\Det(X,\Lambda)$ for a small v-stack $X$ and the six functors with respect to this formalism are defined as in \cite[Definition 1.7]{Sch}.
We define $\Det^{\ULA}(\Hck_{G,S},\Lambda)\subset \Det(\Hck_{G,S},\Lambda)$ as the full subcategory consisting of ULA sheaves, see \cite[\sc{Definition} VI.6.1, \sc{Definition} IV.2.1.]{FS}.
Let 
\[
\Sat(\Hck_{G,S},\Lambda)\subset \Det^{\ULA}(\Hck_{G,S},\Lambda)
\]
denote the Satake category, see \cite[\sc{Definition} VI.7.8]{FS}.
Even if $\Lambda$ is not a torsion ring, we have an ad hoc definition of $\Det(X,\Lambda)$, see \cite[Proposition 26.2]{Sch} and \cite{B}, i.e.
\begin{itemize}
\item If $\Lambda=\mathcal{O}_L$, then $\Det(X,\mathcal{O}_L)=\displaystyle\lim_{\substack{\longleftarrow\\ n}} \Det(X,\mathcal{O}_L/\ell^n)$.
\item If $\Lambda=L$, then $\Det(X,L)=\Det(X,\mathcal{O}_L)[\ell^{-1}]$, which is obtained by inverting $\ell$ in Hom-sets.
\item If $\Lambda=\Qlb$, then $\Det(X,\Qlb)=\displaystyle\lim_{\substack{\longrightarrow\\ L}} \Det(X,L)$.
\end{itemize}
We can define $\Det^{\ULA}(\Hck_{G,S},\Lambda)$ by changing $\Det$ to $D^{\ULA}$ in all the terms in the (co)limits.
Similarly, we define $\Sat(\Hck_{G,S},\Lambda)$.
\subsection{Notation on geometric Satake}
For a v-stack $S$ over $\Div^d_{(\mathcal{Y})}$, there exists a monoidal structure on $\Det(\Hck_{G,S})$ called the convolution product defined as follows:
Put
\begin{align*}
\Hck_{G,S}^{\text{conv}}&:=L^+_{S}G\backslash L_{S}G\times^{L^+_{S}G}L_{S}G/L^+_{S}G\\
&=\Hck_{G,S}\times_{\pi_{2,2},S/L^+_SG,\pi_{1,2}}\Hck_{G,S}.
\end{align*}
There exists a canonical map
\begin{align}\label{eqn:defofconv1}
\mathrm{conv}_{1,S}\colon \Hck_{G,\Spd C}^{\text{conv}}&\to \Hck_{G,S}
\end{align}
induced by the multiplication map $L_SG\times_S L_SG\to L_SG$.
Moreover, there are natural maps
\begin{equation}\label{eqn:defofa'1}
\begin{split}
a'_1\colon \Hck_{G,S}^{\text{conv}}&=\Hck_{G,S}\times_{\pi_{2,2},S/L^+_SG,\pi_{1,2}}\Hck_{G,S}\\
&\to \Hck_{G,S}\times_{S}\Hck_{G,S}.
\end{split}
\end{equation}
For $A,B\in D(\Hck_{G,S},\Lambda)$ The convolution product is defined by 
\[
A\star B:=R(\mathrm{conv}_{1,S})_*a_1^*(A\boxtimes B).
\]

Let
\begin{equation}\label{eqn:defofF}
\begin{split}
F&:=F^d:=F_{G,S}:=F^d_{G,S}\\
&:=\bigoplus_{i\in \mathbb{Z}}\mathcal{H}^i(R\pi_{G,S*})\colon \Sat(\Hck_{G,S})\to \LocSys(S,\Lambda),
\end{split}
\end{equation}
where the functor 
\[
R\pi_{G,S*}\colon \Sat(\Hck_{G,S})\to \Det(S,\Lambda)
\]
is the composition of the pullback to $\Gr_{G,S}$ and pushforward along 
\[
\pi_{G,S}\colon \Gr_{G,S}\to S,
\]
and $\LocSys(S, \Lambda)$ is the full subcategory of $\Det(S,\Lambda)$ consisting of the locally constant sheaves on $S$ whose fibers are finite projective $\Lambda$-modules.
Note that $\LocSys(\Spd k, \Lambda)$ and $\LocSys(\Spd C, \Lambda)$ is equivalent to the category of finite projective $\Lambda$-modules, and $\LocSys(\Div^1,\Lambda)$ is equivalent to the category $\Rep(W_E,\Lambda)$ of continuous representations of $W_E$ on finite projective $\Lambda$-modules.

The two geometric Satake equivalences
\[
\Sat(\Hck_{G,\Spd k},\Qlb)\simeq \Rep(\widehat{G},\Qlb)
\]
in \cite{B} and in \cite{Zhumixed} come from monoidal structures of $F_{G,\Spd k}$ via Tannakian construction, but the way of constructing the monoidal structures of $F_{G,\Spd k}$ is not the same.
\section{Plan of paper}\label{sec:plan}
For Theorem \ref{thm:main}, we need to prove the followings:
\begin{enumerate}
\item[(1)] The associativity constraints
\[
(A\star B)\star C\cong A\star (B\star C)
\]
with respect to the two monoidal structures, are the same.
\item[(2)] The unit constraints
\[
(\delta \star A)\cong A
\]
with respect to the two monoidal structures are the same, where $\delta\in \Sat(\Hck_{G,\Spd k},\Qlb)$ denotes the unit object.
\item[(3)] The commutativity constraints
\[
A\star B\cong B\star A
\]
with respect to the two monoidal structures, are the same.
\item[(4)] The isomorphisms
\[
F(A\star B)\cong F(A)\otimes F(B)
\]
with respect to the two monoidal structures of $F$, are the same.
\end{enumerate}
However, by the faithfulness of $F$, the claims (1), (2), (3) follow from the claim (4).
Thus it suffices to show (4).
In \S \ref{sec:prel}, we show some results as preliminaries.
In \S \ref{sec:twomonoidal}, we show the desired result (4).
\section{Preliminaries}\label{sec:prel}
\subsection{The functor $j_{\Wdisj,(\Spd C)^2}^*$}
Put
\[
(\Spd C)_{\Wdisj}^2:=(\Spd C)^2\setminus (\Spd C\times_{\Div^1}\Spd C).
\]
Since $\Div^1$ is the quotient of $\Spd C$ by the action of the Weil group $W_F$ of $F$ in \cite[\S IV.7]{FS}, $\Spd C\times_{\Div^1}\Spd C\cong \Spd C\times \underline{W_F}$.
The subspace $(\Spd C)_{\Wdisj}^2$ is open in $(\Spd C)^2$ since $\Div^1$ is separated by \cite[\sc{Proposition} II.1.21]{FS}.
Let
\begin{align*}
j_{\Wdisj,(\Spd C)^2}\colon (\Spd C)_{\Wdisj}^2\hookrightarrow (\Spd C)^2
\end{align*}
be the open immersion and let
\[
i_{W_F}\colon (\Spd C)\times \underline{W_F}\hookrightarrow (\Spd C)^2
\]
be the closed immersion.

By the paragraph after Remark VI.2.5 in \cite{FS}, the spaces $\Hck_{G,(\Spd C)^2}$ and $\Gr_{G,(\Spd C)^2}$ split after the base change by $j_{\Wdisj,(\Spd C)^2}$:
{\small
\begin{align*}
\Hck_{G,(\Spd C)^2}\times_{(\Spd C)^2}(\Spd C)_{\Wdisj}^2&\cong (\Hck_{G,\Spd C})^2\times_{(\Spd C)^2}(\Spd C)_{\Wdisj}^2\\
\Gr_{G,(\Spd C)^2}\times_{(\Spd C)^2}(\Spd C)_{\Wdisj}^2&\cong (\Gr_{G,\Spd C})^2\times_{(\Spd C)^2}(\Spd C)_{\Wdisj}^2
\end{align*}
}
To prove the main theorem, we reduce, in a certain sense, to the case over $(\Spd C)^2_{\Wdisj}$ by using $j^*$, where the above splitting is applied. The following lemma is necessary for this reduction:
\begin{lemm}\label{lemm:ffjneq}
Let $A\in \Shet((\Spd C)^2,\Lambda)$ be a constant sheaf of a finitely generated $\Lambda$-module, where $\Shet(-,\Lambda)$ is the full subcategory of $\Det(-,\Lambda)$ consisting of complexes concentrated in degree 0.
Then the natural map
\[
A \to R^0j_{\Wdisj,(\Spd C)^2*}j_{\Wdisj,(\Spd C)^2}^*A
\]
is an isomorphism.
\end{lemm}
\begin{proof}
We may assume that $\Lambda$ is $\ell$-power torsion.
Moreover, by decomposing $A$ and changing $\Lambda$ if necessary, we may assume $A=\Lambda$.
It suffices to show that $R(i_{W_F})^!\Lambda\in \Det^{\geq 2}(\Spd C,\Lambda)$, i.e. the cohomologies of the complex $R(i_{W_F})^!\Lambda$ are nonzero only on degree $\geq 2$.
Choose a pseudo-uniformizer $\varpi\in C^{\flat}$, which induces homomorphism
\[
\overline{\mathbb{F}}_p((x^{p^{-\infty}}))\to C^{\flat}, x\mapsto \varpi,
\]
and a morphism $q\colon \Spd C\to \Spd \overline{\mathbb{F}}_p((x^{p^{-\infty}}))$.
Consider the diagram
{\footnotesize
\[
\xymatrix{
\Spd C\times \underline{W_F}\ar[r]^-{i''}\ar@{=}[rd]&(1\times q)^{-1}(\Spd C\times \underline{W_F})\ar[d]_{q'}\ar[rr]^-{i'}\ar@{}[rrd]|{\square}&&\Spd C\times \Spd C\ar[d]^{1\times q}\\
&\Spd C\times \underline{W_F}\ar[r]_-{i_{W_F}}&\Spd C\times\Spd C\ar[r]_-{1\times q}&\widetilde{\mathbb{D}}^*_{C^{\flat}},
}
\]
}
where $\widetilde{\mathbb{D}}^*_{C^{\flat}}:=\Spd C\times \Spd \overline{\mathbb{F}}_p((x^{p^{-\infty}}))$.
Put $\iota:=(1\times q)\circ i_{W_F}$.
\begin{flushleft}
\textbf{Claim 1.} If $R\iota^!(1\times q)_*\Lambda\in \Det^{\geq 2}(\Spd C\times \underline{W_F})$, then the lemma follows.
\end{flushleft}
\begin{proof}
Assume $R\iota^!(1\times q)_*\Lambda\in \Det^{\geq 2}(\Spd C\times \underline{W_F})$.
Note that $R(1\times q)_{*}\Lambda=(1\times q)_{*}\Lambda \in \Shet(\widetilde{\mathbb{D}}^*_{C^{\flat}},\Lambda)$ by \cite[Remark 21.14]{Sch}.
By the base change, it follows that
\[
R(q')_*R(i')^!\Lambda \cong R\iota^!R(1\times q)_*\Lambda\in \Det^{\geq 2}.
\]
Since the functor $R(q')_*$ is exact (by \cite[Remark 21.14]{Sch}) and faithful (as $q'$ is surjective), it follows that $R(i')^!\Lambda\in \Det^{\geq 2}$, hence $R(i_{W_F})^!\Lambda= R(i'')^! R(i')^!\Lambda \in \Det^{\geq 2}$.
\end{proof}
Now it suffices to show $R\iota^!(1\times q)_*\Lambda\in \Det^{\geq 2}$.
\begin{flushleft}
\textbf{Claim 2.} $R\iota^!\Lambda \in \Det^{\geq 2}$.
\end{flushleft}
\begin{proof}
Write $\alpha$ for the projection $\widetilde{\mathbb{D}}^*_{C^{\flat}}\to \Spd C$.
Since $\alpha$ is $\ell$-cohomologically smooth and 
$R\alpha^!\Lambda$ is in degree $-2$ as in \cite[Proposition 24.1]{Sch} and the proof of \cite[Proposition 24.5]{Sch}.
Moreover, put $q_{W_F}=\alpha \circ \iota$, which is the projection $\Spd C\times \underline{W_F}\to \Spd C$.
Then $Rq_{W_F}^!\Lambda$ is a sheaf written in \cite[Proposition 24.2]{Sch}.
Thus we have $R\iota^!\Lambda\cong R\iota^!(R\alpha^!\Lambda\otimes (R\alpha^!\Lambda)^{-1})\cong Rq_{W_F}^!\Lambda\otimes (R\alpha^!\Lambda)^{-1}\in \Det^{\geq 2}$.
\end{proof}
Hence it follows that
\begin{equation}\label{eqn:!*!geq2}
\begin{split}
R\iota^!(1\times q)_{*}R(1\times q)^{!}\Lambda&\cong R(q')_*R(i')^!R(1\times q)^!\Lambda\\
&\cong R(q')_*R(q')^!R\iota^!\Lambda\\
&\in \Det^{\geq 2}.    
\end{split}
\end{equation}
However, by \textbf{Claim 1}, we have to show $R\iota^!(1\times q)_{*}(1\times q)^{*}\Lambda \in \Det^{\geq 2}$.
Therefore we need to compute the difference between $R\iota^!(1\times q)_{*}R(1\times q)^{!}\Lambda$ and $R\iota^!(1\times q)_{*}(1\times q)^{*}\Lambda$.
\begin{flushleft}
\textbf{Claim 3.} The sheaf $(1\times q)_{*}(1\times q)^{*}\Lambda$ on $\widetilde{\mathbb{D}}_{C^{\flat}}$ injects into $\mathcal{H}^0((1\times q)_{*}R(1\times q)^{!}\Lambda)$.
\end{flushleft}
\begin{proof}
For a finite extension $F'$ of $\overline{\mathbb{F}}_p((x^{p^{-\infty}}))$ in $C^{\flat}$, write $q_{F'}\colon \Spd F'\to \Spd \overline{\mathbb{F}}_p((x^{p^{-\infty}}))$ for the projection.
Then we have
\begin{align}\label{eqn:starstarlimit}
(1\times q)_{*}(1\times q)^*\cong \lim_{\substack{\longrightarrow \\ \text{$F'/\overline{\mathbb{F}}_p((x^{p^{-\infty}}))$:finite}}}(1\times q_{F'})_*(1\times q_{F'})^{*}, 
\end{align}
(see the proof of \cite[Prop 23.4]{Sch}).
Thus
\[
(1\times q)_{*}\Lambda \cong \lim_{\substack{\longrightarrow \\ \text{$F'/\overline{\mathbb{F}}_p((x^{p^{-\infty}}))$:finite}}}(1\times q_{F'})_*\Lambda.
\]
We know that $q_{F'}$ is proper and that $q$ is partially proper as $q$ is the inverse limit of $q_{F'}$.
Thus $R(1\times q_{F'})_*=R(1\times q_{F'})_!$ and $R(1\times q)_*=R(1\times q)_!$ hold.
By passing to the adjoint of (\ref{eqn:starstarlimit}), we have
\begin{align}\label{eqn:upper!limit}
(1\times q)_{*}R(1\times q)^{!}\Lambda \cong \underset{\substack{\longleftarrow \\ \text{$F'/\overline{\mathbb{F}}_p((x^{p^{-\infty}}))$:finite}}}{R\lim}(1\times q_{F'})_*R(1\times q_{F'})^!\Lambda.
\end{align}
Since $q_{F'}$ is \'{e}tale, we have $R(1\times q_{F'})^{!}\Lambda\cong (1\times q_{F'})^{*}\Lambda$.
Therefore, we obtain the natural map 
\begin{align*}
(1\times q)_{*}\Lambda &\cong \lim_{\substack{\longrightarrow \\ \text{$F'/\overline{\mathbb{F}}_p((x^{p^{-\infty}}))$:finite}}}(1\times q_{F'})_*\Lambda\\
&\to \lim_{\substack{\longleftarrow \\ \text{$F'/\overline{\mathbb{F}}_p((x^{p^{-\infty}}))$:finite}}}(1\times q_{F'})_*\Lambda\\
&\cong \lim_{\substack{\longleftarrow \\ \text{$F'/\overline{\mathbb{F}}_p((x^{p^{-\infty}}))$:finite}}}(1\times q_{F'})_*R(1\times q_{F'})^!\Lambda \\
&\cong \mathcal{H}^0((1\times q)_{*}R(1\times q)^{!}\Lambda),
\end{align*}
which is injective since it is a sheafification of an injection of presheaves from a direct limit as presheaves to an inverse limit.
\end{proof}
Put
\[
Q:=\mathcal{H}^0((1\times q)_{*}R(1\times q)^!\Lambda)/(1\times q)_{*}(1\times q)^{*}\Lambda.
\]
By the long exact sequence, \textbf{Claim 1} and (\ref{eqn:!*!geq2}), we have to show that
\[
R\iota^!Q\in \Det^{\geq 1}.
\]
It suffices to show that $\Hom_{\Det(\Spd C\times \underline{W_F},\Lambda)}(A,R\iota^!Q)=0$ for any $A\in \Shet(\Spd C\times \underline{W_F},\Lambda)$, equivalently,
\[
\Hom_{\Shet(\widetilde{\mathbb{D}}^*_{C^{\flat}},\Lambda)}(\iota_*A,Q)=0.
\]
Here we use the equality $\iota_*=R\iota_*=R\iota_!$ since $1\times q$ is partially proper and $R(1\times q)_*$ is exact by \cite[Remark 21.14]{Sch}.

Since $\Spd C\times \underline{W_F}$ is a locally spatial diamond, \cite[Proposition 14.15]{Sch} says that $\Det(\Spd C\times \underline{W_F},\Lambda)$ is the left-completion of $D((\Spd C\times \underline{W_F})_{\et},\Lambda)$, the derived category of $\Lambda$-sheaves on \'{e}tale site over $\Spd C\times \underline{W_F}$.
The same holds for $\widetilde{\mathbb{D}}^*_{C^{\flat}}$.
Also, $\iota_*=\iota_{\et*}$ holds, where $\iota_{\et*}$ is the pushforward functor for sheaves on \'{e}tale sites.

Thus it suffices to show the following claim:
\begin{flushleft}
\textbf{Claim 4.}
\[
\Hom_{\Sh((\widetilde{\mathbb{D}}^*_{C^{\flat}})_{\et},\Lambda)}(\iota_*A,Q)=0
\]
for any $A\in \Sh((\Spd C\times \underline{W_F})_{\et},\Lambda)$.
\end{flushleft}
\begin{proof}
Assume 
\[
f\in \Hom_{\Sh((\widetilde{\mathbb{D}}^*_{C^{\flat}})_{\et},\Lambda)}(\iota_*A,Q)
\]
is nonzero, that is, there exists a geometric point $x\colon \Spd C'\to \widetilde{\mathbb{D}}^*_{C^{\flat}}$ such that the stalk $f_x$ is nonzero.
There exist an \'{e}tale map $U\to \widetilde{\mathbb{D}}^*_{C^{\flat}}$ and $a\in (\iota_{*}A)(U)$ such that the point $\Spd C'\overset{x}{\to} \widetilde{\mathbb{D}}^*_{C^{\flat}}$ factors through $U$ and $f(a)_x\neq 0$.
By shrinking $U$, there exists 
\[
\widetilde{f(a)}\in (\lim_{\substack{\longleftarrow \\ F'}}(1\times q_{F'})_*\Lambda)(U)\overset{(\ref{eqn:upper!limit})}{\cong}  \mathcal{H}^0((1\times q)_{*}R(1\times q)^!\Lambda)(U)
\]
such that its image in $Q(U)$ is $f(a)(U)$.
For each finite extension $F'$ over $\overline{\mathbb{F}}_p((t^{p^{-\infty}}))$ in $C^{\flat}$, consider the diagram
\[
\xymatrix{
U_C\ar[r]\ar[d]\ar@{}[rd]|{\square}&U_{F'}\ar[r]\ar[d]\ar@{}[rd]|{\square}&U\ar[d]\\
(\Spd C)^2\ar[r]&\Spd C\times \Spd F'\ar[r]_-{1\times q_{F'}}&\widetilde{\mathbb{D}}^*_{C^{\flat}}.
}
\]
Let $V$ be a connected component of $U$ whose image in $\widetilde{\mathbb{D}}^*_{C^{\flat}}$ contains the image of $x$.
Since $U_{F'}\to U$ is finite \'{e}tale surjective, there exists a connected component $V_{F'}$ of $U_{F'}$ which restricts to a surjection $V_{F'}\twoheadrightarrow V$.
We can choose $\{V_{F'}\}_{F'}$ functorially in $F'$.
Now put
\[
V_C:=\lim_{\substack{\longleftarrow \\ F'}}V_{F'}.
\]
The space $V_C$ is connected and the map $V_C\to V$ is surjective by construction.
Let $\widetilde{x}$ be a geometric point of $V_C$ whose image in $\widetilde{\mathbb{D}}^*_{C^{\flat}}$ is $x$.
Consider the element
\[
(1\times q)^*\widetilde{f(a)}\in ((1\times q)^*(\lim_{\substack{\longleftarrow \\ F'}}(1\times q_{F'})_*\Lambda))(U_C).
\]
The map $\Spd C\to \Spd \overline{\mathbb{F}}_p((t^{p^{-\infty}}))$ is universally open since it is the quotient map by an action of the absolute Galois group of $\overline{\mathbb{F}}_p((t^{p^{-\infty}}))$.
Therefore the map $1\times q$ is universally open.
From this we can show that the morphism of sheaves
\[
(1\times q)^*(\lim_{\substack{\longleftarrow \\ F'}}(1\times q_{F'})_*\Lambda)\to \lim_{\substack{\longleftarrow \\ F'}}(1\times q)^*(1\times q_{F'})_*\Lambda
\]
is an isomorphism since the section of both sheaves on an \'{e}tale map $U'\to (\Spd C)^2$ is equal to $\displaystyle\lim_{\substack{\longleftarrow \\ F'}}\Lambda((1\times q_{F'})^{-1}(1\times q)(U))$.

As $(f(a))_x$ is nonzero, we have
\begin{align*}
((1\times q)^*\widetilde{f(a)})_{\widetilde{x}}&\notin (1\times q)^*(\lim_{\substack{\longrightarrow \\ F'}}(1\times q_{F'})_*\Lambda)_{\widetilde{x}}\\
&=(\lim_{\substack{\longrightarrow \\ F'}}(1\times q)^*(1\times q_{F'})_*\Lambda)_{\widetilde{x}}.
\end{align*}
Since the \'{e}tale morphism $1\times q_{F'}$ splits after a base change by $1\times q$, the sheaf $(1\times q)^*(1\times q_{F'})_*\Lambda$ is a constant sheaf of a finitely generated $\Lambda$-module.
Thus the sheaf $\displaystyle\lim_{\substack{\longleftarrow \\ F'}}(1\times q)^*(1\times q_{F'})_*\Lambda$ is a sheaf of the form $W'\mapsto C(|W'|,K)$ for some profinite space $K$.
Since $V_C$ is connected but $K$ is totally disconnected, it follows that 
\[
((1\times q)^*\widetilde{f(a)})_{y}\notin (\lim_{\substack{\longrightarrow \\ F'}}(1\times q)^*(1\times q_{F'})_*\Lambda)_{y}
\]
for any $y\in V_C$.
Hence $(f(a))_y\neq 0$ for any $y\in V_C$.
However, let $U_C'$ be the fiber product as in
\[
\xymatrix{
U_C'\ar[r]\ar[d]\ar@{}[rd]|{\square}&U_C\ar[d]\\
\Spd C\times \underline{W_F}\ar[r]_-{\iota}&\widetilde{\mathbb{D}}^*_{C^{\flat}}.
}
\]
Then $U_C'$ is pro\'{e}tale over $\Spd C\times \underline{W_F}$, in particular totally disconnected, but $V_C$ is connected and not a single point (as $V_C\to V$ is surjective).
Therefore we have a point $y\in V_C\setminus \mathrm{Im}(U_C')$.
Since the image of $y$ is away from $\Spd C\times \underline{W_F}$, it follows that
\[
f_y\in \Hom((\iota_{*}A)_y,Q_y)=0,
\]
which is a contradiction.
\end{proof}
\end{proof}
Let $\mathrm{Sh}_{\et}^{\mathrm{consta}}((\Spd C)^2,\Lambda)$ be the full subcategory of $\mathrm{Sh}_{\et}((\Spd C)^2,\Lambda)$ consisting of constant sheaves of finitely generated $\Lambda$-modules.
By the standard argument, Lemma \ref{lemm:ffjneq} implies the following corollary:
\begin{cor}\label{cor:ffjWdisj}
The pullback functor
\[
j_{\Wdisj}^*\colon \mathrm{Sh}_{\et}^{\mathrm{consta}}((\Spd C)^2,\Lambda)\to \Det((\Spd C)^2_{\Wdisj},\Lambda)
\]
is fully faithful.
\end{cor}
Let $\Det^{\mathrm{consta}}((\Spd C)^n,\Lambda)\subset \Det((\Spd C)^n,\Lambda)$ denote the full subcategory consisting of constant complexes with perfect fibers, i.e. the full subcategory of objects isomorphic to the pullback of a perfect complex in $\Det(*,\Lambda)\simeq D(\Lambda\mathrm{\text{-mod}})$.
This full subcategory satisfies the following lemma:
\begin{lemm}\label{lemm:Dconsta}
The full subcategory $\Det^{\mathrm{consta}}((\Spd C)^n,\Lambda)\subset \Det((\Spd C)^n,\Lambda)$ is a Serre subcategory.
Moreover, the functor 
\[
H^0((\Spd C)^n,-)
\]
is exact on $\Det^{\mathrm{consta}}((\Spd C)^n,\Lambda)$.
\end{lemm}
\begin{proof}
Both statements reduce to the case $n=0$ by using the fact that the pullback functor
\[
\Det(*,\Lambda)\to \Det((\Spd C)^n,\Lambda)
\]
is fully faithful, see \cite[Theorem 1.13(ii)]{Sch}.
\end{proof}
\begin{cor}\label{cor:ffjneq}
For $A\in \Det^{\mathrm{consta}}((\Spd C)^2,\Lambda)$, the homomorphism
\[
j_{\Wdisj,(\Spd C)^2}^*\colon H^*((\Spd C)^2,A)\to H^*((\Spd C)^2_{\Wdisj} ,j_{\Wdisj,(\Spd C)^2}^*A)
\]
is injective.
\end{cor}
\begin{proof}
It suffices to show the injectivity on degree $0$.
Then we may assume that $A$ is concentrated in degree $\leq 0$.
We have the following homomorphism of exact sequences:
{\footnotesize
\[
\xymatrix@!C=90pt{
H^0((\Spd C)^2, \tau_{\leq -1}A)\ar[r]\ar[d]^{j^*}&H^0((\Spd C)^2, A)\ar[r]\ar[d]^{j^*}&H^0((\Spd C)^2, \mathcal{H}^0(A))\ar[d]^{j^*}\\
H^0((\Spd C)^2_{\Wdisj}, j^*\tau_{\leq -1}A)\ar[r]&H^0((\Spd C)^2_{\Wdisj}, j^*A)\ar[r]&H^0((\Spd C)^2_{\Wdisj}, j^*\mathcal{H}^0(A)),
}
\]
}
where $j=j_{\Wdisj,(\Spd C)^2}$.
Lemma \ref{lemm:Dconsta} implies $H^0((\Spd C)^2, \tau_{\leq -1}A)=0$, and Lemma \ref{lemm:ffjneq} implies the right vertical arrow is an isomorphism.
This implies the desired injectivity.
\end{proof}
We write $D_{\mathrm{lc}}(-,\Lambda)\subset \Det(-,\Lambda)$ for the full subcategory of locally constant complexes with perfect fibers.
We also need the following lemma later:
\begin{lemm}\label{lemm:Divconsta}
 If $A\in \Det((\Spd C)^n,\Lambda)$ is in the essential image of the pullback functor
 \[
 D_{\mathrm{lc}}((\Div^1)^n,\Lambda)\to \Det((\Spd C)^n,\Lambda),
 \]
 then $A\in \Det^{\mathrm{consta.}((\Spd C)^n,\Lambda)}$.
 \end{lemm}
 \begin{proof}
 The claim follows from a commutative diagram of pullback functors
 \[
 \xymatrix{
 D_{\mathrm{lc}}([*/W_E]^n)\ar[r]\ar[d]& D_{\mathrm{lc}}((\Div^1)^n)\ar[d]\\
 D_{\mathrm{lc}}(*)\ar[r]& D_{\mathrm{lc}}((\Spd C)^n), 
 }
 \]
 and the fact that the upper horizontal arrow is an equivalence by \cite[\sc{Proposition} IV.7.3]{FS}.
 \end{proof}
\subsection{The v-stack $\Hck^{(i)}$ and the fusion product}
In this subsection, we define the following v-stack $\Hck^{(i)}$ and v-sheaf $\Gr^{(i)}$ ($i=1,2$) over $(\Spd C)^2$:
\begin{defi}
\begin{enumerate}
\item For $i=1,2$, the v-stack $\Hck^{(i)}_{G,(\Div^1_{\mathcal{Y}})^2}$ over $(\Div^1_{\mathcal{Y}})^2$ is a functor sending an affinoid perfectoid $S\to (\Div^1_{\mathcal{Y}})^2$ to the groupoid of pairs of $G$-bundles $\mathcal{E}_1,\mathcal{E}_2$ over $B^+_{\Div^2_{\mathcal{Y}}}(S)$ (with respect to the composition $S\to (\Div^1_{\mathcal{Y}})^2\to \Div^2_{\mathcal{Y}}$) together with an isomorphism $\phi$ between $\mathcal{E}_1$ and $\mathcal{E}_2$ over $B^+_{\Div^2_{\mathcal{Y}}}(S)[1/\mathcal{I}_i]$.
Here $\mathcal{I}_i$ is an ideal sheaf corresponding to a map
\begin{align}
S\to (\Div^1_{\mathcal{Y}})^2\overset{\mathrm{pr}_i}{\to} \Div^1_{\mathcal{Y}},
\label{map:ithdiv}
\end{align}
where $\mathrm{pr}_i$ is the $i$-th projection.
\item For $i=1,2$, the v-stack $\Gr^{(i)}_{G,(\Div^1_{\mathcal{Y}})^2}$ over $(\Div^1_{\mathcal{Y}})^2$ is a functor sending an affinoid perfectoid $S\to (\Div^1_{\mathcal{Y}})^2$ to the sets of a $G$-bundle $\mathcal{E}_1$ over $B^+_{\Div^2_{\mathcal{Y}}}(S)$ together with a trivialization $\phi$ of $\mathcal{E}_1$ over $B^+_{\Div^2_{\mathcal{Y}}}(S)[1/\mathcal{I}_i]$.
\end{enumerate}
\end{defi}
We can write $\Hck^{(i)}$ and $\Gr^{(i)}$ in terms of a variant of loop groups:
\begin{defi}
For $i=1,2$, the v-sheaf $L^{(i)}_{(\Div^1_{\mathcal{Y}})^2}G$ over $(\Div^1_{\mathcal{Y}})^2$ is a v-sheaf given by
\[
S\mapsto G(B_{\Div^2_{\mathcal{Y}}}^+(S)[1/\mathcal{I}_i]).
\]
\end{defi}
\begin{prop}
\begin{enumerate}
\item There is a natural isomorphism of \'{e}tale stacks over $(\Div^1_{\mathcal{Y}})^2$
\[
\Hck^{(i)}_{G,(\Div^1_{\mathcal{Y}})^2}\cong (L^+_{(\Div^1_{\mathcal{Y}})^2}G)\backslash (L^{(i)}_{(\Div^1_{\mathcal{Y}})^2}G)/(L^+_{(\Div^1_{\mathcal{Y}})^2}G).
\]
\item There is a natural isomorphism of \'{e}tale sheaves over $(\Div^1_{\mathcal{Y}})^2$
\[
\Gr^{(i)}_{G,(\Div^1_{\mathcal{Y}})^2}\cong (L^{(i)}_{(\Div^1_{\mathcal{Y}})^2}G)/(L^+_{(\Div^1_{\mathcal{Y}})^2}G).
\]
\end{enumerate}
\end{prop}
\begin{proof}
This follows from the same argument as \cite[\sc{Proposition} VI.1.7, VI.1.9]{FS}.
\end{proof}
For $i=1,2$, let $B^+_i(S)$ be a ring $B^+_{\Div^1_{\mathcal{Y}}}(S)$ with respect to a natural map 
(\ref{map:ithdiv}), that is, the completion of $\mathcal{O}_{\mathcal{Y}_S}$ along $\mathcal{I}_i$.
Then there is a natural map
\begin{align}\label{eqn:defofpi}
p_i\colon \Hck^{(i)}_{G,(\Div^1_{\mathcal{Y}})^2}\to \Hck_{G,\Div^1_{\mathcal{Y}}}
\end{align}
sending $(\mathcal{E}_1,\mathcal{E}_2, \phi)$ to its pullback $(\mathcal{E}_1|_{B^+_i(S)},\mathcal{E}_2|_{B^+_i(S)}, \phi|_{B^+_i(S)[1/\mathcal{I}_i]})$ under a natural map 
\[
B^+_i(S)\to B^+_{\Div^2_{\mathcal{Y}}}(S),
\]
noting that the right hand side is the completion of $\mathcal{O}_{\mathcal{Y}_S}$ along $\mathcal{I}_i$ and the left is along $\mathcal{I}_1\mathcal{I}_2$.
Also, there are two projections
\begin{equation}\label{eqn:defofpii2}
\begin{split}
\pi_{1,2},\pi_{2,2}\colon \Hck^{(i)}_{G,(\Div^1_{\mathcal{Y}})^2}&\cong (L^+_{(\Div^1_{\mathcal{Y}})^2}G)\backslash (L^{(i)}_{(\Div^1_{\mathcal{Y}})^2}G)/(L^+_{(\Div^1_{\mathcal{Y}})^2}G)\\
&\to (\Div^1_{\mathcal{Y}})^2/(L^+_{(\Div^1_{\mathcal{Y}})^2}G)
\end{split}
\end{equation}
with respect to the left and right actions of $L^+_{(\Div^1_{\mathcal{Y}})^2}G$ on $L^{(i)}_{(\Div^1_{\mathcal{Y}})^2}G$, respectively.

For a v-stack $S\to (\Div^1_{\mathcal{Y}})^2$, put
\[
\Hck^{(i)}_{G,S}:=\Hck^{(i)}_{G,(\Div^1_{\mathcal{Y}})^2}\times_{(\Div^1_{\mathcal{Y}})^2}S,
\]
and still write $\pi_{1,2},\pi_{2,2}$ for the base change of $\pi_{1,2},\pi_{2,2}$ to $S$.
We also write $p_i$ for the map
\begin{align*}
p_i\times \mathrm{pr}_i\colon \Hck^{(i)}_{G,(\Spd C)^2}&=\Hck^{(i)}_{G,(\Div^1_{\mathcal{Y}})^2}\times_{(\Div^1_{\mathcal{Y}})^2}(\Spd C)^2\\
&\to \Hck_{G,\Div^1_{\mathcal{Y}}}\times_{\Div^1_{\mathcal{Y}}}\Spd C=\Hck_{G,\Spd C}.    
\end{align*}
Put
\begin{align*}
\Hck_{G,(\Spd C)^2}^{\text{conv}}&:=L^+_{(\Spd C)^2}G\backslash L^{(1)}_{(\Spd C)^2}G\times^{L^+_{(\Spd C)^2}G}L^{(2)}_{(\Spd C)^2}G/L^+_{(\Spd C)^2}G,\\
&=\Hck^{(1)}_{G,(\Spd C)^2}\times_{\pi_{2,2},(\Spd C)^2/L^+G,\pi_{1,2}}\Hck^{(2)}_{G,(\Spd C)^2}.
\end{align*}
There exists a canonical map
\begin{align}\label{eqn:defofconv2}
\mathrm{conv}_2\colon \Hck_{G,(\Spd C)^2}^{\text{conv}}&\to \Hck_{G,(\Spd C)^2}
\end{align}
induced by the multiplication map 
\[
L^{(1)}_{(\Spd C)^2}G\times_{(\Spd C)^2}L^{(2)}_{(\Spd C)^2}G\to L_{(\Spd C)^2}G.
\]
Moreover, there are natural maps
\begin{equation}\label{eqn:defofa'2}
\begin{split}
a'_2\colon \Hck_{G,(\Spd C)^2}^{\text{conv}}&=\Hck^{(1)}_{G,(\Spd C)^2}\times_{\pi_{2,2},(\Spd C)^2/L^+G,\pi_{1,2}}\Hck^{(2)}_{G,(\Spd C)^2}\\
&\to \Hck^{(1)}_{G,(\Spd C)^2}\times_{(\Spd C)^2}\Hck^{(2)}_{G,(\Spd C)^2}.
\end{split}
\end{equation}
For $A,B\in \Sat(\Hck_{G,\Spd C},\Qlb)$, define the fusion product 
\[
A\ast B\in \Sat(\Hck_{G,(\Spd C)^2},\Qlb)
\]
by
\[
A\ast B:= R(\mathrm{conv}_2)_* a_2^{\prime,*}(p_1^*A\boxtimes p_2^*B).
\]
We will use the following lemma later:
\begin{lemm}\label{lemm:FAastBconst}
Let $A,B\in \Sat(\Hck_{G,\Spd C},\Qlb)$.
The sheaf $F^2_{\Spd C}(A\ast B)\in \mathrm{LocSys}((\Spd C)^2,\Lambda)$ (defined in (\ref{eqn:defofF})) is constant.
\end{lemm}
\begin{proof}
The category $\Sat(\Hck_{G,\Spd C},\Qlb)$ is semisimple and its simple objects are the intersection sheaves, which admit a Weil descent.
Thus $A$ and $B$ can be written as the pullbacks of sheaves in $\Sat(\Hck_{G,\Div^1},\Qlb)$.
It holds that
\[
F^2(A\ast B)=\bigoplus_{i}R^i(\pi_{G,\Spd C})_*\pi_{\Gr_{\Spd C}}^*R(\mathrm{conv}_2)_*(a_2^{\prime}\circ (p_1\times p_2))^*(A\boxtimes B),
\]
where $\pi_{\Gr_{\Spd C}}\colon \Gr_{G,\Spd C}\to \Hck_{G,\Spd C}$ is the natural projection.
The morphisms appearing here can also be defined over $\Div^1$, and $\pi_{G,\Div^1}$ and $\mathrm{conv}_{2,\Div^1}$ is ind-proper.
This implies that $F^2(A\ast B)$ can be written as the pullback of an object in  $\mathrm{LocSys}((\Div^1)^2,\Lambda)$.
Then the claim follows from Lemma \ref{lemm:Divconsta}.
\end{proof}
\section{The two monoidal structures}\label{sec:twomonoidal}
In this section, we will prove the main theorem, which states that two isomorphisms
\[
F(A\star B)\cong F(A)\otimes F(B)
\]
agree for $A,B\in \Sat(\Hck_{G,\Spd k},\Qlb)$.

We start with some preliminaries and proceed to review the first monoidal structure: Zhu's monoidal structure.
Next, we explain the second monoidal structure, which needs some arguments.
Finally, we prove the main theorem.
\subsection{Preliminaries}
Let 
\begin{align*}
i_{\Spd \mathcal{O}_C/L^+G}\colon \Spd k/L^+_{\Spd k}G\to \Spd \mathcal{O}_C/L^+_{\Spd \mathcal{O}_C}G,\\
j_{\Spd \mathcal{O}_C/L^+G}\colon \Spd C/L^+_{\Spd C}G\to \Spd \mathcal{O}_C/L^+_{\Spd \mathcal{O}_C}G
\end{align*}
be the natural immersions.
We also write 
\begin{align*}
i_{\Hck}\colon \Hck_{G,\Spd k}\to \Hck_{G,\Spd \mathcal{O}_C},\\
j_{\Hck}\colon \Hck_{G,\Spd C}\to \Hck_{G,\Spd \mathcal{O}_C}
\end{align*}
for the natural immersions.
For a small v-stack $S\to \Div^1_{(\mathcal{Y})}$, the two maps
\[
\pi_{1,1,S},\pi_{2,1,S}\colon \Hck_{G,S}\cong [L^+_SG\backslash L_SG/L^+_SG]\to [S/L^+_SG],
\]
with respect to the left and right actions, induce an $H^*([S/L^+_SG],\Lambda)$-bialgebra structure on $H^*(\Hck_{G,S},\Lambda)$.
Thus for $A\in \Det^{\ULA}(\Hck_{G,S},\Lambda)$, the module $H^*(\Hck_{G,S},A)$ is an $H^*([S/L^+_SG],\Lambda)$-bimodule.
\begin{lemm}\label{lemm:kC}
\begin{enumerate}
\item The pullback homomorphisms
\begin{align*}
H^*(\Spd \mathcal{O}_C/L^+_{\Spd \mathcal{O}_C}G,\Lambda)\to H^*(\Spd k/L^+_{\Spd k}G,\Lambda),\\
H^*(\Spd \mathcal{O}_C/L^+_{\Spd \mathcal{O}_C}G,\Lambda)\to H^*(\Spd C/L^+_{\Spd C}G,\Lambda)
\end{align*}
induced by $i_{\Spd \mathcal{O}_C/L^+G}$ and $j_{\Spd \mathcal{O}_C/L^+G}$, respectively, are isomorphisms of rings.
\item Put $R_{G}:=H^*(\Spd k/L^+_{\Spd k}G,\Lambda)$.
For $A\in \Det^{\ULA}(\Hck_{G,\Spd \mathcal{O}_C},\Lambda)$ and $S=\Spd C$, $\Spd k$ or $\Spd \mathcal{O}_C$, 
the cohomology group 
\[
H^*(\Hck_{G,S},A)
\]
is a bimodule over $H^*([S/L^+_SG],S)\cong R_G$ by (i).
Then the pullback homomorphisms
\begin{align*}
H^*(\Hck_{G,\Spd \mathcal{O}_C},A)\to H^*(\Hck_{G,\Spd k},i_{\Hck}^*A),\\
H^*(\Hck_{G,\Spd \mathcal{O}_C},A)\to H^*(\Hck_{G,\Spd C},j_{\Hck}^*A)
\end{align*}
are isomorphisms of $R_G$-bimodules.
\end{enumerate}
\end{lemm}
\begin{proof}
\begin{enumerate}
\item
There are equivalences of categories
\newlength{\len}
\settowidth{\len}{$\Det^{\ULA}(\Spd k/L^+_{\Spd k}G)$}
\[
\xymatrix@C=50pt@R=3pt{
\Det^{\ULA}(\Spd k/L^+_{\Spd k}G)\ar@{<-}[r]^-{i_{\Spd \mathcal{O}_C/L^+G}^*}&\Det^{\ULA}(\Spd \mathcal{O}_C/L^+_{\Spd \mathcal{O}_C}G)\\
\hspace{\len}\ar[r]^-{j_{\Spd \mathcal{O}_C/L^+G}^*}&\Det^{\ULA}(\Spd C/L^+_{\Spd C}G),
}
\]
which can be obtained by restricting the result of \cite[\sc{Corollary} VI.6.7]{FS} to the $0$-Schubert cell $\Hck_{G,S,0}$.
Since the constant sheaf $\Lambda$ is ULA, we have
\begin{align*}
&H^m(\Spd \mathcal{O}_C/L^+_{\Spd \mathcal{O}_C}G,\Lambda)\\
&=\Hom_{\Det(\Spd \mathcal{O}_C/L^+_{\Spd \mathcal{O}_C}G)}(\Lambda,\Lambda[m])\\
&\cong \Hom_{\Det(\Spd k/L^+_{\Spd k}G)}(i_{\Spd \mathcal{O}_C/L^+G}^*\Lambda,i_{\Spd \mathcal{O}_C/L^+G}^*\Lambda[m])\\
&=H^m(\Spd k/L^+_{\Spd k}G,\Lambda).
\end{align*}
This homomorphism is equal to the natural map in the statement, and it is a homomorphism of rings by the general theory.
The proof of the second isomorphism is the same.
\item
When we regard $H^*(\Hck_{G,\Spd k},i_{\Hck}^*A)$ as a $H^*(\Hck_{G,\Spd \mathcal{O}_C},\Lambda)$-module via the pullback homomorphism 
\[
H^*(\Hck_{G,\Spd \mathcal{O}_C},\Lambda)\to H^*(\Hck_{G,\Spd k},\Lambda),
\]
the first homomorphism in the statement is a homomorphism of $H^*(\Hck_{G,\Spd C},\Lambda)$-modules.
Thus, when we consider both sides of this homomorphism as a bimodule over 
\[
R_G\overset{\text{(i)}}{\cong} H^*([\Spd \mathcal{O}_C/L^*_{\Spd \mathcal{O}_C}G],\Lambda)
\]
via the two maps $\pi_{1,1,\Spd \mathcal{O}_C}$ and $\pi_{2,1,\Spd \mathcal{O}_C}$, this homomorphism is a homomorphism of $R_G$-bimodules.
This $R_G$-bimodule structure on the target agrees with the bimodule structure in the statement since the diagram
\[
\xymatrix@C=40pt{
\Hck_{G,\Spd \mathcal{O}_C}\ar[r]^-{\pi_{i,1,\Spd \mathcal{O}_C}}&\Spd \mathcal{O}_C/L^+_{\Spd \mathcal{O}_C}G\\
\Hck_{G,\Spd k}\ar[u]\ar[r]^-{\pi_{i,1,\Spd k}}&\Spd k/L^+_{\Spd k}G\ar[u]
}
\]
commutes for $i=1,2$.
Therefore, the first homomorphism in the statement is a homomorphism of $R_G$-bimodules.

Now we can show the claim by the same argument as (1), using equivalences of categories
\begin{align*}
\Det^{\ULA}(\Hck_{G,\Spd k})&\xleftarrow{i_{\Hck}^*}\Det^{\ULA}(\Hck_{G,\Spd \mathcal{O}_C})\\
&\xrightarrow{j_{\Hck}^*}\Det^{\ULA}(\Hck_{G,\Spd C})    
\end{align*}
in \cite[\sc{Corollary} VI.6.7]{FS} and noting that $\Lambda_{\mathrm{supp}(A)}$ is ULA by \cite[\sc{Proposition} VI.6.5]{FS}.
\end{enumerate}
\end{proof}
\subsection{The first monoidal structure: Zhu's monoidal structure}\label{ssc:Zhu}
In \cite[\S 2.3]{Zhumixed}, they construct the monoidal structure as follows: (here, by `monoidal structure' we only mean an isomorphism $F(A\star B)\cong F(A)\otimes F(B)$ in view of \S \ref{sec:plan}.)

Note that we use the identification of several schemes or stacks over $\Spec k$ with corresponding diamonds or v-stacks over $\Spd k$ via \cite[\S 27]{Sch}.
\begin{enumerate}
\item For $A,B\in \Perv(\Hck_{G,\Spd k},\Qlb)$, the natural map
\begin{align*}
&R\Gamma(\Hck_{G,\Spd k},A)\otimes_{\Qlb}R\Gamma(\Hck_{G,\Spd k},B)\\
&\overset{-\boxtimes -}{\longrightarrow} R\Gamma(\Hck_{G,\Spd k}\times_{\Spd k}\Hck_{G,\Spd k},A\boxtimes B)\\
&\overset{(a'_1)^*}{\longrightarrow} R\Gamma(\Hck_{G,\Spd k}^{\mathrm{conv}},(a'_1)^*(A\boxtimes B))\\
&=R\Gamma(\Hck_{G,\Spd k},A\star B)
\end{align*}
induces an isomorphism
\begin{align}\label{eqn:Zhuequivmonoidal}
H^*(\Hck_{G,\Spd k},A)\otimes_{R_G}H^*(\Hck_{G,\Spd k},B)\cong H^*(\Hck_{G,\Spd k},A\star B)
\end{align}
of $R_G$-bimodules. 
\item \label{item:twoaction} Two actions of $R_G$ on $H^*(\Hck_{G,\Spd k},A)$ coincide for any perverse sheaf $A\in \Perv(\Hck_{G,\Spd k},\Qlb)$.
\item By tensoring $\Qlb$ over $R_G$ with the isomorphism (\ref{eqn:Zhuequivmonoidal}) and using (\ref{item:twoaction}) the fact that 
\[
\Qlb\otimes_{R_G} H^*(\Hck_{G,\Spd k},A)\cong H^*(\Gr_{G,\Spd k},A)
\]
holds, we have the natural isomorphism
\begin{align*}
H^*(\Gr_{G,\Spd k},A)\otimes H^*(\Gr_{G,\Spd k},B)\cong H^*(\Gr_{G,\Spd k},A\star B).
\end{align*}
\end{enumerate}
Using Lemma \ref{lemm:kC}, the same argument as above works if we replace $\Spd k$ by $\Spd C$.
\subsection{Construction of the second monoidal strucutre}\label{ssc:FS}
In this subsection, we explain the construction of the second monoidal structure of 
\[
F^1_{\Spd k}\colon \Sat(\Hck_{G,\Spd k},\Qlb)\to \mathrm{LocSys}(\Spd k,\Qlb)=\mathrm{Vect}_{\Qlb}.
\]
In \cite{B}, we implicitly use the result of this subsection.
For the construction, we first construct a monoidal structure of
\[
F^1_{\Spd C}\colon \Sat(\Hck_{G,\Spd C},\Qlb)\to \mathrm{LocSys}(\Spd C,\Qlb)=\mathrm{Vect}_{\Qlb},
\]
by an argument similar to \cite[\sc{Definition/Proposition} VI.9.4.]{FS}.
More precisely, it is as follows:
Recall that in \cite{FS}, they construct a monoidal structure of 
\begin{align}\label{eqn:FSmonoidal}
F^1_{\Div^1}\colon \Sat(\Hck_{G,\Div^1},\Lambda)\to \mathrm{LocSys}(\Div^1,\Lambda)=\Rep(W_F,\Lambda)
\end{align}
for a torsion $\Lambda$.
The same argument works for $\Lambda=\Qlb$, so we obtain a monoidal structure of 
\begin{align*}\label{eqn:FSmonoidalQlb}
F^1_{\Div^1}\colon \Sat(\Hck_{G,\Div^1},\Qlb)\to \mathrm{LocSys}(\Div^1,\Lambda)=\Rep(W_F,\Qlb).
\end{align*}
In this subsection, we will show the following proposition:
\begin{prop}\label{prop:F1monoidal}
There exists a symmetric monoidal structure 
\[
F^1_{\Spd C}\colon \Sat(\Hck_{G,\Spd C},\Qlb)\to \mathrm{LocSys}(\Spd C,\Lambda)
\]
such that the natural isomorphism which makes the diagram
\begin{align}\label{eqn:diag:F1Div1SpdC}
\vcenter{
\xymatrix{
\Sat(\Hck_{G,\Div^1},\Qlb)\ar[r]^{F^1_{\Div^1}}\ar[d]_{(-)^*}&\mathrm{LocSys}(\Div^1,\Lambda)\ar[d]^{(-)^*}\\
\Sat(\Hck_{G,\Spd C},\Qlb)\ar[r]^{F^1_{\Spd C}}&\mathrm{LocSys}(\Spd C,\Spd C)
}
}
\end{align}
commute is symmetric monoidal with respect to the above symmetric monoidal structure of $F^1_{\Div^1}$ and the canonical symmetric monoidal structures of the pullback functors.
\end{prop}
Once we prove this proposition, we obtain a monoidal structure of $F^1_{\Spd k}$ by connecting $\Spd C$ with $\Spd k$ via \cite[\sc{Corollary} VI.6.7]{FS} (or \cite{B}.)
\begin{proof}
The proposition follows since we can give a parallel construction to Fargue--Scholze's construction of (\ref{eqn:FSmonoidal}), noting that in the part where they used \cite[Proposition VI.9.3]{FS}, we need to use Corollary \ref{cor:ffjneq} (and Lemma \ref{lemm:FAastBconst}) instead.

Namely, it is as follows:
One can easily show that there are cartesian squares
\[
\xymatrix{
\Hck_{G,\Spd C}\ar@{<-}[r]^-{\mathrm{conv}_1}\ar[d]_{\Delta_{\Hck}}&\Hck^{\mathrm{conv}}_{G,\Spd C}\ar[r]^-{a'_1}\ar[d]^{\Delta_{\Hck^{\mathrm{conv}}}}&\Hck_{G,\Spd C}\times_{\Spd C}\Hck_{G,\Spd C}\ar[d]^{\Delta_{\Hck^2}}&\\
\Hck_{G,(\Spd C)^2}\ar@{<-}[r]^-{\mathrm{conv}_2}&\Hck^{\mathrm{conv}}_{G,(\Spd C)^2}\ar[r]^-{a'_2}&\Hck^{(1)}_{G,(\Spd C)^2}\times_{(\Spd C)^2}\Hck^{(2)}_{G,(\Spd C)^2},
}
\]
where the maps $\mathrm{conv}_1$ and $a_1'$ are defined in (\ref{eqn:defofconv1}) and (\ref{eqn:defofa'1}), respectively, and the vertical maps are the base changes of the diagonal map $\Spd C\to (\Spd C)^2$, under the isomorphism
\[
\Hck_{G,\Spd C}^{\text{conv}}\cong \Hck_{G,(\Spd C)^2}^{\text{conv}}\times_{(\Spd C)^2,\Delta}\Spd C, 
\]
and similar isomorphisms.
Thus, the restriction of $A\ast B$ to the diagonal is isomorphic to the convolution $A\star B$.
On the other hand, the restriction of $A\ast B$ to 
\[
\Hck_{G,(\Spd C)^2_{\Wdisj}}\cong (\Hck_{G,\Spd C})^2\times_{(\Spd C)^2}(\Spd C)^2_{\Wdisj}
\]
is isomorphic to the restriction of $A\boxtimes B$.
Also, the K\"{u}nneth formula defines an isomorphism
\begin{align}\label{eqn:FSKunneth}
F^2_{(\Spd C)^2_{\Wdisj}}((A\boxtimes B)|_{\Hck_{G,(\Spd C)^2_{\Wdisj}}})\cong (F^1(A)\boxtimes F^1(B))|_{(\Spd C)^2_{\Wdisj}}.
\end{align}
Therefore
\[
F^2_{(\Spd C)^2}(A\ast B)|_{\Hck_{G,(\Spd C)^2_{\Wdisj}}}\cong (F^1(A)\boxtimes F^1(B))|_{(\Spd C)^2_{\Wdisj}}.
\]
All objects in $\Det(\Spd C,\Lambda)$ are constant (since $\Spd C$ is strictly totally disconnected and $|\Spd C|=*$).
This implies that $F^1(A)$ and $F^1(B)$ are constant, so $F^1(A)\boxtimes F^1(B)$ are constant as well.
Moreover, $F^2_{(\Spd C)^2}(A\ast B)$ is constant by Lemma \ref{lemm:FAastBconst}.
By Corollary \ref{cor:ffjneq}, we have an isomorphism
\begin{align}\label{eqn:FSF2F1F1}
F^2_{(\Spd C)^2}(A\ast B)\cong F^1_{\Spd C}(A)\boxtimes F^1_{\Spd C}(B)
\end{align}
and 
\[
H^*(\Gr_G,A\star B)\cong H^*(\Gr_G,A)\otimes H^*(\Gr_G,B)
\]
by restricting to the diagonal.
The other conditions for being symmetric monoidal can be shown in exactly the same way, parallel to the argument in 
\cite{FS}.
The condition concerning the square (\ref{eqn:diag:F1Div1SpdC}) follows easily since this construction is completely parallel to Fargues--Scholze's one.
\end{proof}
\subsection{Proof of main theorem}
Now we can show the main theorem:
\begin{thm}\label{thm:main:re}
For $A, B\in \Sat(\Hck_{G,\Spd k},\Qlb)$, the isomorphism
\[
F(A\star B)\cong F(A)\otimes F(B)
\]
in \S \ref{ssc:FS} coincides with one in \S \ref{ssc:Zhu}.
\end{thm}
As in \S \ref{ssc:FS}, the second isomorphism comes from a monoidal structure on
\[
F^1_{\Spd C}\colon \Sat(\Hck_{G,\Spd C},\Qlb)\to \mathrm{Vect}_{\Qlb},
\]
by connecting $\Spd C$ with $\Spd k$ via \cite[\sc{Corollary} VI.6.7]{FS}.
On the other hand, as written in the last paragraph in \S \ref{ssc:Zhu}, there is a method of constructing a monoidal structure of $F^1_{\Spd C}$, which is parallel to Zhu's construction.
This construction is compatible with Zhu's one via \cite[\sc{Corollary} VI.6.7]{FS} in view of Lemma \ref{lemm:kC}.
Therefore, it suffices to show that the two monoidal structures of $F^1_{\Spd C}$ coincide:
\begin{thm}
The above two isomorphisms
\[
F^1_{\Spd C}(A\star B)\cong F^1_{\Spd C}(A)\otimes F^1_{\Spd C}(B)
\]
coincide.
\end{thm}
\begin{proof}
We have the following commutative diagram in which all the squares are cartesian:
\begin{align}
\vcenter{
\xymatrix{
\Hck_{G,\Spd C}\times_{\Spd C} \Hck_{G,\Spd C}
\ar[r]^-{\Delta'_{\Hck^2}}
\ar@{=}[d]
&
\Hck_{G,\Spd C}\times\Hck_{G,\Spd C}
\ar@{<-}[d]^{p_1\times p_2}
\\
\Hck_{G,\Spd C}\times_{\Spd C} \Hck_{G,\Spd C}
\ar[r]^-{\Delta_{\Hck^2}}
\ar@{<-}[d]_{a_1'}
&
\Hck_{G,(\Spd C)^2}^{(1)}\times_{(\Spd C)^2}\Hck_{G,(\Spd C)^2}^{(2)}
\ar@{<-}[d]_{a_2'}
\\
\Hck^{\mathrm{conv}}_{G,\Spd C}
\ar[r]^-{\Delta_{\Hck^{\mathrm{conv}}}}
\ar[d]_{\mathrm{conv}_1}
&
\Hck^{\mathrm{conv}}_{G,(\Spd C)^2}
\ar[d]^{\mathrm{conv}_2}
\\
\Hck_{G,\Spd C}
\ar[r]^-{\Delta_{\Hck}}
\ar@{<-}[d]_{\pi_{\Gr_{\Spd C}}}
&
\Hck_{G,(\Spd C)^2}
\ar@{<-}[d]_{\pi_{\Gr_{(\Spd C)^2}}}
\\
\Gr_{G,\Spd C}
\ar[r]^-{\Delta_{\Gr}}
&
\Gr_{G,(\Spd C)^2}.
}
}\label{eqn:diagmainDelta}
\end{align}
Here the horizontal arrows are the pullback of the diagonal map 
\[
\Delta_{(\Spd C)^2}\colon \Spd C\to (\Spd C)^2
\]
and $\pi_{\Gr_{\Spd C}}, \pi_{\Gr_{(\Spd C)^2}}$ are the projections.
See (\ref{eqn:defofpi}), (\ref{eqn:defofa'1}), (\ref{eqn:defofa'2}) for the definition of $p_i$, $a_1'$, $a_2'$, respectively.
We also have the following commutative diagram in which all the squares are cartesian:
{\small
\begin{align}
\vcenter{
\xymatrix{
\Hck_{G,\Spd C}\times\Hck_{G,\Spd C}
\ar@{<-}[r]^{j'_{\Hck^2}}
\ar@{<-}[d]^{p_1\times p_2}
&
(\Hck_{G,\Spd C})^2\times_{(\Spd C)^2}(\Spd C)^2_{\Wdisj}
\ar@{=}[d]
\\
\Hck_{G,(\Spd C)^2}^{(1)}\times_{(\Spd C)^2}\Hck_{G,(\Spd C)^2}^{(2)}
\ar@{<-}[r]^{j_{\Hck^2}}
\ar@{<-}[d]_{a_2'}
&
(\Hck_{G,\Spd C})^2\times_{(\Spd C)^2}(\Spd C)^2_{\Wdisj}
\ar@{=}[d]
\\
\Hck^{\mathrm{conv}}_{G,(\Spd C)^2}
\ar@{<-}[r]^-{j_{\Hck^{\mathrm{conv}}}}
\ar[d]^{\mathrm{conv}_2}
&
(\Hck_{G,\Spd C})^2\times_{(\Spd C)^2}(\Spd C)^2_{\Wdisj}
\ar@{=}[d]
\\
\Hck_{G,(\Spd C)^2}
\ar@{<-}[r]^-{j_{\Hck}}
\ar@{<-}[d]_{\pi_{\Gr_{(\Spd C)^2}}}
&
(\Hck_{G,\Spd C})^2\times_{(\Spd C)^2}(\Spd C)^2_{\Wdisj}
\ar@{<-}[d]^{\pi_{\Gr_{(\Spd C)^2_{\Wdisj}}}}
\\
\Gr_{G,(\Spd C)^2}
\ar@{<-}[r]^-{j_{\Gr}}
&
(\Gr_{G,\Spd C})^2\times_{(\Spd C)^2}(\Spd C)^2_{\Wdisj}.
}
}\label{eqn:diagmainWdisj}
\end{align}
}
Here the horizontal arrows are the pullback of the immersion 
\[
j_{\Wdisj}\colon (\Spd C)^2_{\Wdisj}\to (\Spd C)^2.
\]

As reviewed in \S \ref{ssc:Zhu}, the first monoidal structure is defined as follows:
The canonical map
\begin{equation}\label{eqn:xstary}
\begin{split}
&R\Gamma(\Hck_{G,\Spd C},A)\otimes_{\Qlb}R\Gamma(\Hck_{G,\Spd C},B)\\
&\overset{-\boxtimes-}{\longrightarrow} R\Gamma(\Hck_{G,\Spd C}\times_{\Spd C}\Hck_{G,\Spd C},A\boxtimes B)\\
&\overset{(a'_1)^*}{\longrightarrow} R\Gamma(\Hck_{G,\Spd C}^{\mathrm{conv}},(a'_1)^*(A\boxtimes B))\\
&\cong R\Gamma(\Hck_{G,\Spd C},A\star B)
\end{split}
\end{equation}
induces an isomorphism
\[
H^*(\Hck_{G,\Spd C},A)\otimes_{R_G}H^*(\Hck_{G,\Spd C},B)\cong H^*(\Hck_{G,\Spd C},A\star B).
\]
By applying $\Qlb\otimes_{R_G}-$, we have
\begin{multline}\label{fml:Zhu}
\Qlb\otimes_{R_G}H^*(\Hck_{G,\Spd C},A)\otimes_{R_G}H^*(\Hck_{G,\Spd C},B)\\\cong H^*(\Gr_{G,\Spd C},A\star B).
\end{multline}
On the other hand, we have the second monoidal structure:
\begin{align}\label{fml:FS}
H^*(\Gr_{G,\Spd C},A)\otimes_{\Qlb}H^*(\Gr_{G,\Spd C},B)\cong H^*(\Gr_{G,\Spd C},A\star B), 
\end{align}
as reviewed in \S \ref{ssc:FS}.
It suffices to show that the composition
\begin{equation}\label{fml:compo}
\begin{split}
&\Qlb\otimes_{R_G}H^*(\Hck_{G,\Spd C},A)\otimes_{R_G}H^*(\Hck_{G,\Spd C},B)\\
&\overset{(\ref{fml:Zhu})}{\cong} H^*(\Gr_{G,\Spd C},A\star B)\\
&\overset{(\ref{fml:FS})}{\cong} H^*(\Gr_{G,\Spd C},A)\otimes_{\Qlb}H^*(\Gr_{G,\Spd C},B)
\end{split}
\end{equation}
is a natural map under the fact that the two $R_G$-actions on $H^*(\Hck_{G,\Spd C},A)$ coincide.
That is, when we write $\overline{x}$ for the image of $x\in H^*(\Hck_{G,(\Spd C)^n},F)$ under the natural map $H^*(\Hck_{G,(\Spd C)^n},F)\to H^*(\Gr_{G,(\Spd C)^n},F)$, it is enough to show that the composition (\ref{fml:compo}) is the map of the form
\[
1\otimes x\otimes y\mapsto \overline{x}\otimes \overline{y}.
\]
First, for $x\in H^i(\Hck_{G,\Spd C},A)$ and $y\in H^j(\Hck_{G,\Spd C},B)$, we can define an element $x\ast y$ of
\begin{align*}
R^{i+j}\Gamma(\Hck_{G,(\Spd C)^2},A\ast B), 
\end{align*}
as the image of $x\otimes y$ under a map
\begin{align*}
&R\Gamma(\Hck_{G,\Spd C},A)\otimes_{\Qlb} R\Gamma(\Hck_{G,\Spd C},B)\\
&\overset{-\boxtimes -}{\longrightarrow} R\Gamma(\Hck_{G,\Spd C}\times \Hck_{G,\Spd C}, A\boxtimes B)\\
&\overset{(a'_2)^*(p_1\times p_2)^*}{\longrightarrow} R\Gamma(\Hck^{\mathrm{conv}}_{G,(\Spd C)^2},(a'_2)^*(p_1^*A\boxtimes p_2^*B))\\
&\cong R\Gamma(\Hck_{G,(\Spd C)^2},R(\mathrm{conv}_{2})_*(a'_2)^*(p_1^*A\boxtimes p_2^*B))\\
&=R\Gamma(\Hck_{G,(\Spd C)^2},A\ast B).
\end{align*}
\textbf{Claim 1.} Under the isomorphism (\ref{fml:Zhu}), $\Delta_{\Gr}^*(\overline{x\ast y})$ corresponds to 
\[
1\otimes x\otimes y\in \Qlb\otimes_{R_G}H^*(\Hck_{G,\Spd C},A)\otimes_{R_G}H^*(\Hck_{G,\Spd C},B).
\]
\begin{proof}
It suffices to show that the homomorphism (\ref{eqn:xstary}) maps $x\otimes y$ to $\Delta_{\Hck}^*(x\ast y)$.
This can be proved by a standard argument using the cartesian squares in (\ref{eqn:diagmainDelta}) and applying the proper base change for the ind-proper morphism $\mathrm{conv}_2$, noting that all concerning sheaves have bounded supports. 
\end{proof}
On the other hand, we can define an element $\overline{x}\boxtimes \overline{y}$ of 
\[
H^{i+j}((\Gr_{G,\Spd C})^2,A\boxtimes B),
\]
as the image of $x\otimes y$ under a map
\begin{align*}
&H^i(\Hck_{G,\Spd C},A)\otimes_{\Qlb}H^j(\Hck_{G,\Spd C},B)\\
&\overset{\mathrm{pullback}}{\longrightarrow} H^i(\Gr_{G,\Spd C},A)\otimes_{\Qlb}H^j(\Gr_{G,\Spd C},B)\\
&\overset{-\boxtimes -}{\longrightarrow} H^{i+j}((\Gr_{G,\Spd C})^2,A\boxtimes B).
\end{align*}
Let $j_{\Gr}'\colon (\Gr_{\Spd C})^2\times_{(\Spd C)^2}(\Spd C)^2_{\Wdisj}\to (\Gr_{\Spd C})^2$ be the immersion.\\
\textbf{Claim 2.} The restrictions $(j_{\Gr})^*(\overline{x\ast y})$ and $(j'_{\Gr})^*(\overline{x}\boxtimes \overline{y})$ of $\overline{x\ast y}$ and $\overline{x}\boxtimes \overline{y}$ to $(\Spd C)^2_{\Wdisj}$ coincide as an element of the $(i+j)$-th cohomology of 
\begin{align*}
&R\Gamma((\Gr_{\Spd C})^2\times_{(\Spd C)^2}(\Spd C)^2_{\Wdisj}, (j'_{\Gr})^*(A\boxtimes B)).
\end{align*}
\begin{proof}
This can be proved by the standard argument using the diagram (\ref{eqn:diagmainWdisj}) and base changes.
\end{proof}
By Lemma \ref{lemm:FAastBconst} and Lemma \ref{lemm:Dconsta}, we have
\begin{align*}
H^{i+j}(\Gr_{G,(\Spd C)^2}, A\ast B)
&\cong H^0((\Spd C)^2,R^{i+j}(\pi_{G,(\Spd C)^2})_*(A\ast B))\\
&\subset H^0((\Spd C)^2,F^2(A\ast B)).
\end{align*}
Thus the element $\overline{x\ast y}$ can be seen as an element of $H^0((\Spd C)^2,F^2(A\ast B))$.
Moreover, we obtain a commutative diagram
{\footnotesize
\[
\xymatrix@!C=150pt{
H^{i+j}(\Gr_{G,(\Spd C)^2}, A\ast B)
\ar[r]^{\sim}
\ar[d]_{j^*}
&H^0((\Spd C)^2,R^{i+j}\pi_*(A\ast B))
\ar[d]^{j^*}
\\
H^{i+j}(\Gr_{G,(\Spd C)^2_{\Wdisj}}, j^*(A\ast B))
\ar[r]
\ar[d]_{\sim}
&H^0((\Spd C)^2_{\Wdisj},j^*R^{i+j}\pi_*(A\ast B))
\ar[d]^{\sim}
\\
H^{i+j}(\Gr_{G,(\Spd C)^2_{\Wdisj}}, j^*(A\boxtimes B))
\ar[r]
\ar@{<-}[d]_{j^*}
&H^0((\Spd C)^2_{\Wdisj},j^*\displaystyle\bigoplus_{i'+j'=i+j}R^{i'}\pi_*A\boxtimes R^{j'}\pi_*B))
\ar@{<-}[d]^{j^*}
\\
H^{i+j}((\Gr_{G,\Spd C})^2, A\boxtimes B)\ar[r]
\ar@{<-}[d]_{-\boxtimes -}
&H^0((\Spd C)^2,\displaystyle\bigoplus_{i'+j'=i+j}R^{i'}\pi_*A\boxtimes R^{j'}\pi_*B))
\ar@{<-}[d]^{-\boxtimes -}
\\
H^{i}(\Gr_{G,\Spd C},A)\otimes H^j(\Gr_{G,\Spd C},B)\ar[r]^{\sim}&H^0(\Spd C,R^{i}\pi_*A)\otimes H^0(\Spd C,R^{i}\pi_*B).
}
\]
}
Here we only write $j$ for various base changes of $(\Spd C)^2_{\Wdisj}\to (\Spd C)^2$, and $\pi$ for the projections $\Gr_{G,\Spd C}\to \Spd C$ and $\Gr_{G,(\Spd C)^2}\to (\Spd C)^2$.
The five objects on the right are by definition direct summands of 
\begin{align*}
&H^0((\Spd C)^2, F^2(A\ast B)), H^0((\Spd C)^2, j^*F^2(A\ast B)), \\
&H^0((\Spd C)^2, j^*(F^1(A)\boxtimes F^1(B))),
H^0((\Spd C)^2, F^1(A)\boxtimes F^1(B))),\\
&\text{and }H^0(\Spd C,F^1(A))\otimes H^0(\Spd C,F^1(B)),
\end{align*}
respectively.
The second vertical arrow from the top on the right is induced by the isomorphism (\ref{eqn:FSF2F1F1}).

Let $(\overline{x}\boxtimes \overline{y})_{(\Spd C)^2}$ be the image of $\overline{x}\boxtimes \overline{y}$ under the homomorphism
\begin{align*}
&H^i(\Gr_{G,\Spd C},A)\otimes H^j(\Gr_{G,\Spd C},B)\\
&\to H^0(\Spd C,R^i\pi_*A)\otimes H^0(\Spd C,R^j\pi_*B)\\
&\subset H^0(\Spd C,F^1(A))\otimes H^0(\Spd C,F^1(B))\\
&\overset{-\boxtimes -}{\to}H^0((\Spd C)^2,F^1(A)\boxtimes F^1(B)).
\end{align*}
By the above diagram and \textbf{Claim 2}, it follows that the restriction of $\overline{x\ast y}\in H^0((\Spd C)^2, F^2(A\ast B))$ to $H^0((\Spd C)^2_{\Wdisj},j^*F^2(A\ast B))$
is equal to the restriction of $(\overline{x}\boxtimes \overline{y})_{(\Spd C)^2}$ to $H^0((\Spd C)^2,j^*(F^1(A)\boxtimes F^1(B)))$ under the isomorphism (\ref{eqn:FSF2F1F1}).
Thus by Corollary \ref{cor:ffjneq}, 
\[
\Delta_{(\Spd C)^2}^*(\overline{x\ast y})\in H^0(\Spd C, \Delta^*(F^2(A\ast B)))\cong F^1(A\ast B)
\]
is equal to 
\[
\Delta_{(\Spd C)^2}^*((\overline{x}\boxtimes \overline{y})_{(\Spd C)^2})\in H^0(\Spd C, \Delta^*(F^1(A)\boxtimes F^2(B)))\cong F^1(A)\otimes F^1(B)
\]
under (\ref{fml:FS}).
It is easy to show that the restriction of $(\overline{x}\boxtimes \overline{y})_{(\Spd C)^2}$ to 
\begin{align*}
H^0(\Spd C,\Delta_{(\Spd C)^2}^*(F^1(A)\boxtimes F^1(B)))&\cong H^0(\Spd C,F^1(A)\otimes F^1(B))\\
&\cong H^*(\Gr_G,A)\otimes H^*(\Gr_G,B).
\end{align*}
is $\overline{x}\otimes \overline{y}$.
This implies that $\Delta^*_{\Gr}(\overline{x\ast y})$ corresponds to $\overline{x}\otimes \overline{y}$ under the isomorphism (\ref{fml:FS}).
From this and \textbf{Claim 1}, it follows that the map (\ref{fml:compo}) can be written as
\[
1\otimes x\otimes y\mapsto \overline{x}\otimes \overline{y},
\]
and the theorem follows.
\end{proof}
\bibliographystyle{test}
\bibliography{twomonoidalbib}
\end{document}